\documentclass[12pt]{article}
\usepackage{amssymb,amsthm,amsmath}
\usepackage{fixltx2e}
\usepackage[T1]{fontenc}
\usepackage[british]{babel}
\usepackage[all]{xy}
\usepackage{txfonts}
\usepackage{graphicx}
\usepackage{hyperref}
\usepackage{mathbbol}
\usepackage{mathabx}
\usepackage{calrsfs}
\usepackage[a4paper,text={128mm,185mm}]{geometry}

\theoremstyle{plain}
\newtheorem{theorem}[subsection]{Theorem}
\newtheorem{lemma}[subsection]{Lemma}
\newtheorem{proposition}[subsection]{Proposition}

\theoremstyle{definition}
\newtheorem{example}[subsection]{Example}
\newtheorem{remark}[subsection]{Remark}

\newcommand{\comp}{}
\newcommand{\defn}{\textbf}

\newcommand{\ot}{\leftarrow}
\renewcommand{\implies}{$\Rightarrow$}

\newcommand{\coeq}{\ensuremath{\mathrm{Coeq}}}
\newcommand{\C}{\ensuremath{\mathcal{C}}}

\newcommand{\Grpd}{\ensuremath{\mathsf{Grpd}}}
\newcommand{\Span}{\ensuremath{\mathsf{Span}}}
\newcommand{\Cat}{\ensuremath{\mathsf{Cat}}}
\newcommand{\MG}{\ensuremath{\mathsf{MG}}}
\newcommand{\PreGrpd}{\ensuremath{\mathsf{PreGrpd}}}
\newcommand{\RG}{\ensuremath{\mathsf{RG}}}
\newcommand{\V}{\ensuremath{\mathcal{V}}}
\newcommand{\W}{\ensuremath{\mathcal{W}}}
\newcommand{\N}{\ensuremath{\mathbb{N}}}
\newcommand{\noproof}{\hfill \qed}

\newcommand{\matrixone}{\xymatrix@!@1@=2em{{\cdot} && {\cdot}\\
& {\cdot} \ar[lu]^-{x_{1}} \ar[ld]_-{x_{2}} \ar[ru]_-{y_{1}} \ar[rd]^-{y_{2}} && {\cdot} \ar[lu]^-{z_{1}} \ar[ld]_-{z_{2}}\\
{\cdot} && {\cdot}\\
& {\cdot} \ar[lu]^-{x_{3}} \ar[ru]_-{y_{3}} && {\cdot} \ar[lu]^-{z_{3}}}}

\newcommand{\matrixtwo}{\xymatrix@!@1@=2em{{\cdot} && {\cdot}\\
& {\cdot}  && {\cdot} \ar[lllu]|-{p(x_1,y_1,z_{1})}\ar[llld]|-{p(x_2,y_2,z_{2})} \\
{\cdot} && {\cdot}\\
& {\cdot}  && {\cdot} \ar[lllu]|-{p(x_1,y_1,z_{1})}}}

\newcommand{\matrixthree}{\xymatrix@!@1@=2em{{\cdot} && {\cdot}\\
& {\cdot}  && {\cdot}  \\
{\cdot} && {\cdot}\\
& {\cdot} \ar[luuu]^-{p(x_1,x_2,x_{3})}\ar[ruuu]|-{p(y_1,y_2,y_{3})} && {\cdot} \ar[luuu]_-{p(z_1,z_2,z_{3})}}}

\newcommand{\matrixfour}{\xymatrix@!@1@=2em{{\cdot} && {\cdot}\\
& {\cdot}  && {\cdot}  \\
{\cdot} && {\cdot}\\
& {\cdot}  && {\cdot} \ar@/_2ex/[llluuu]|-{} \ar@/^2ex/[llluuu]|-{}}}

\def\pullback{
 \ar@{-}[]+R+<6pt,-1pt>;[]+RD+<6pt,-6pt>%
 \ar@{-}[]+D+<1pt,-6pt>;[]+RD+<6pt,-6pt>}

\newdir{>}{{}*:(1,-.2)@^{>}*:(1,+.2)@_{>}}
\newdir{<}{{}*:(1,+.2)@^{<}*:(1,-.2)@_{<}}

\title{\footnotetext{The first author was supported by IPLeiria/ESTG-CDRSP and Funda\c c\~ao para a Ci\^encia e a Tecnologia (under grants number SFRH/BPD/43216/2008, PTDC/EME-CRO/120585/2010 and PTDC/MAT/120222/2010). The second author is a Research
Associate of the Fonds de la Recherche Scientifique--FNRS. His research was supported by Centro de Matem\'atica da Universidade de Coimbra (CMUC) and by Funda\c c\~ao para a Ci\^encia e a Tecnologia (under grant number SFRH/BPD/38797/2007). He wishes to thank the Instituto Poli\-t\'ecnico for its kind hospitality during his stay in Leiria.}}

\author{\uppercase{\bf Categories vs.\ groupoids\\ \bf via generalised Mal'tsev properties}\\[.5cm]
\it by Nelson MARTINS-FERREIRA and Tim VAN DER LINDEN
}

\begin{document}

\date{}

\maketitle

\vspace{1cm}

\begin{minipage}{118mm}{\small
\textbf{R\'esum\'e.} On \'etudie la diff\'erence entre les cat\'egories internes et les groupo\"ides internes en termes de propri\'et\'es de Malcev g\'en\'eralis\'ees---la propri\'et\'e de Malcev faible d'un c\^ot\'e, et l'$n$-permutabilit\'e de l'autre. Dans la premi\`ere partie de l'article on donne des conditions sur les structures cat\'egoriques internes qui d\'etectent si la cat\'egorie ambiante est naturellement de Malcev, de Malcev ou faiblement de Malcev. On d\'emontre que celles-ci ne d\'ependent pas de l'existence de produits binaires. Dans la seconde partie on se concentre sur les vari\'et\'es d'alg\`ebres universelles.

\smallskip
\textbf{Abstract.} We study the difference between internal categories and internal groupoids in terms of generalised Mal'tsev properties---the weak Mal'tsev property on the one hand, and $n$-permutability on the other. In the first part of the article we give conditions on internal categorical structures which detect whether the surrounding category is naturally Mal'tsev, Mal'tsev or weakly Mal'tsev. We show that these do not depend on the existence of binary products. In the second part we focus on varieties of algebras. 


}\end{minipage}
\thispagestyle{empty}
\pagestyle{empty}

\section*{Introduction}
In this article we study the difference between internal categories and internal groupoids through the generalised Mal'tsev properties their surrounding category may have---the weak Mal'tsev property on the one hand, and $n$-permutability on the other. Conversely, or equivalently, we try to better understand these Mal'tsev conditions by providing new characterisations and new examples for them, singling out distinctive properties of a given type of category via properties of its internal categorical structures: internal categories, (pre)groupoids, relations.

The first part of the text gives a conceptual unification of three levels of Mal'tsev properties: \emph{naturally Mal'tsev} categories~\cite{Johnstone:Maltsev} using groupoids, categories, pregroupoids, etc.\ (Theorem~\ref{th1.nat}), \emph{Mal'tsev} categories~\cite{Carboni-Kelly-Pedicchio, Carboni-Pedicchio-Pirovano} using equivalence relations, preorders, difunctional relations (Theorem~\ref{th4.mal}), and \emph{weakly Mal'tsev} categories~\cite{NMF1} via strong equivalence relations, strong preorders, difunctional strong relations (Theorem~\ref{th6.weak}). Each of the resulting collections of equivalent conditions is completely parallel to the others, and such that a weaker collection of conditions is characterised by a smaller class of internal structures. 

Some of these characterisations are well established, whereas some others are less familiar; what is new in all cases is the context in which we prove them: we never use binary products, but restrict ourselves to categories in which kernel pairs and split pullbacks exist.

The notion of weakly Mal'tsev category is probably not as well known as the others. It was introduced in~\cite{NMF1} as a setting where any internal reflexive graph admits at most one structure of internal category. It turned out that this new notion is weaker than the concept of Mal'tsev category. But, unlike in Mal'tsev categories, in this setting not every internal category is automatically an internal groupoid. This gave rise to the following problem: to characterise those weakly Mal'tsev categories in which internal categories and internal groupoids coincide.

In Section~\ref{Section-Cat-vs-Grpd} we observe that, in a weakly Mal'tsev category with kernel pairs and equalisers, the following hold: (1) the forgetful functor from internal categories to multiplicative graphs is an isomorphism; (2) the forgetful functor from internal groupoids to internal categories is an isomorphism if and only if every internal preorder is an equivalence relation (Theorem~\ref{th8.main}). 

We study some varietal implications of this result in Section~\ref{Section-Varieties}. In finitary quasivarieties of universal algebra, the latter condition---that reflexivity and transitivity together imply symmetry---is known to be equivalent to the variety being \emph{$n$-permutable}, for some~$n$ (Proposition~\ref{Proposition-Permutable}). On the way we recall Proposition~\ref{Proposition-n-Permutable}, a result due to Hagemann~\cite{Hagemann-Mitschke}---see also the monograph \cite{CEL}, and the article~\cite{JRVdL1} where it is proved in the context of regular categories. We furthermore explain how to construct a weakly Mal'tsev quasivariety starting from a Goursat (= $3$-permutable) quasivariety (Proposition~\ref{Proposition-Subvariety}), and use this procedure to show that categories which are both weakly Mal'tsev and Goursat still need not be Mal'tsev (Example~\ref{All-but-Maltsev}).

Of course, via part (2) of Theorem~\ref{th8.main}, our Proposition~\ref{Proposition-Permutable} implies that, in an $n$-permutable weakly Mal'tsev variety, every internal category is an internal groupoid---but surprisingly, here in fact the weak Mal'tsev property is not needed: $n$-permutability suffices, as was recently proved by Rodelo~\cite{Rodelo:Internal-categories} and further explored in the paper~
\cite{MFRVdL1}. This indicates that there may still be hidden connections between these two (a priori independent) weakenings of the Mal'tsev axiom.

\section{Preliminaries}\label{Section-Preliminaries}
We recall the definitions and basic properties of some internal categorical structures which we shall use throughout this article.

\subsection{Split pullbacks}
Let $\C$ be any category. A diagram in $\C$ of the form
\begin{equation}\label{splitsquare}
\vcenter{\xymatrix@!0@=4em{
E \ar@<.5ex>[r]^-{p_2} \ar@<-.5ex>[d]_-{p_1} & C \ar@<.5ex>[l]^-{e_2}
\ar@<-.5ex>[d]_-{g}
 \\
A
 \ar@<.5ex>[r]^-{f} \ar@<-.5ex>[u]_-{e_1}
& B
 \ar@<.5ex>[l]^-{r} \ar@<-.5ex>[u]_-{s}
 }}
\end{equation}
such that
\[
g p_2 =f p_1,\quad p_1 e_2 =rg,\quad e_1 r = e_2 s,\quad p_2 e_1 =sf
\]
and
\[
p_1 e_1=1_{A},\quad fr =1_{B},\quad gs =1_{B}, \quad p_2 e_2 =1_{C}
\]
is called a \defn{double split epimorphism}. When we call a double split epimorphism a \defn{pullback} we refer to the commutative square of split epimorphisms $fp_{1}=gp_{2}$. Any pullback of a split epimorphism along a split epimorphism gives rise to a double split epimorphism; we say that $\C$ \defn{has split pullbacks} when the pullback of a split epimorphism along a split epimorphism always exists.

In a category with split pullbacks $\C$, any diagram such as
\begin{equation}\label{couniv}
\vcenter{\xymatrix@!0@=4em{A \ar@<.5ex>[r]^-{f} \ar[rd]_-{\alpha} & B
\ar@<.5ex>[l]^-{r}
\ar@<-.5ex>[r]_-{s}
\ar[d]^-{\beta} & C \ar@<-.5ex>[l]_-{g} \ar[ld]^-{\gamma}\\
& D}}
\end{equation}
where $fr=1_{B}=gs$ and $\alpha r=\beta=\gamma s$ induces a diagram
\begin{equation}\label{kite}
\vcenter{\xymatrix@!0@=3em{ & C \ar@<.5ex>[ld]^-{e_2} \ar@<-.5ex>[rd]_-{g}
\ar@/^/[rrrd]^-{\gamma} \\
A\times_{B}C \ar@<.5ex>[ru]^-{\pi_2}
\ar@<-.5ex>[rd]_-{\pi_1} && B \ar@<.5ex>[ld]^-{r} \ar@<-.5ex>[lu]_-{s}
 \ar[rr]|-{\beta} && D\\
& A \ar@<.5ex>[ru]^-{f} \ar@<-.5ex>[lu]_-{e_1} \ar@/_/[urrr]_-{\alpha}}}\end{equation}
in which the square is a double split epimorphism. This kind of diagram will appear in the statements of Theorem~\ref{th1.nat}, \ref{th4.mal} and~\ref{th6.weak} as part of a universal property: under certain conditions we expect it to induce a (unique) morphism $\varphi\colon{A\times_{B}C\to D}$ such that $\varphi e_{1}=\alpha$ and $\varphi e_{2}=\gamma$.

\subsection{Internal groupoids}
A \defn{reflexive graph} in $\C$ is a diagram of the form
\begin{equation}\label{reflgraph}
\xymatrix@!0@=4em{C_1 \ar@<1ex>[r]^-{d} \ar@<-1ex>[r]_-{c} & C_0 \ar[l]|-{e}}
\end{equation}
such that $de=1_{C_{0}}=ce$.

A \defn{multiplicative graph} in $\C$ is a diagram of the form
\begin{equation}\label{multgraph}
\xymatrix@!0@=4em{
C_2
\ar@<3ex>[r]^-{\pi_2}
\ar@<-3ex>[r]_-{\pi_1}
\ar[r]|{m}
& C_1
\ar@<-1.5ex>[l]|-{e_2}
\ar@<1.5ex>[l]|-{e_1}
\ar@<1.5ex>[r]^-{d}
\ar@<-1.5ex>[r]_-{c}
& C_0 \ar[l]|-{e}
}
\end{equation}
where
\[
m e_1=1_{C_1}=m e_2,\quad dm =d\pi_2\quad\text{and}\quad cm = c\pi_1
\]
and the double split epimorphism
\begin{equation*}
\xymatrix@!0@=4em{
C_2
\ar@<.5ex>[r]^-{\pi_2} \ar@<-.5ex>[d]_-{\pi_1}
& C_1
\ar@<.5ex>[l]^-{e_2} \ar@<-.5ex>[d]_-{c}
 \\
C_1
 \ar@<.5ex>[r]^-{d} \ar@<-.5ex>[u]_-{e_1}
& C_0
 \ar@<.5ex>[l]^-{e} \ar@<-.5ex>[u]_-{e}
 }
\end{equation*}
is a pullback. Observe that a multiplicative graph is in particular a reflexive
graph ($de=1_{C_{0}}=ce$) and that the morphisms $e_1$ and $e_2$ are universally induced by the pullback:
\[
e_1 = \langle 1_{C_{1}},ed\rangle \quad\text{and}\quad
e_2 =\langle ec,1_{C_{1}}\rangle.
\]
When the category $\C$ admits split pullbacks we shall refer to a multiplicative graph simply as
\begin{equation*}
\xymatrix@!0@=4em{
C_2
\ar[r]^-{m}
& C_1
\ar@<1ex>[r]^-{d}
\ar@<-1ex>[r]_-{c}
& C_0. \ar[l]|-{e}}
\end{equation*}

An \defn{internal category} is a multiplicative graph which satisfies the associativity condition $m(1\times m)=m(m\times 1)$.

An \defn{internal groupoid} is an internal category where both squares $dm=d\pi_2$ and $cm=c\pi_1$ are pullbacks (see for instance~\cite[Proposition~A.3.7]{Borceux-Bourn}). Equivalently, there should be a morphism $t\colon C_1\to C_1$ with $ct=d$, $dt=c$ and $m\langle 1_{C_{1}},t\rangle =ec$, $m\langle t, 1_{C_{1}}\rangle =ed$.

In the following sections we shall consider the obvious forgetful functors
\begin{equation*}
\xymatrix{\Grpd(\C) \ar[r]^-{U_3} & \Cat(\C) \ar[r]^-{U_2} & \MG(\C) \ar[r]^-{U_1} & \RG(\C)}
\end{equation*}
from groupoids in $\C$ to internal categories, to multiplicative graphs, to reflexive graphs. We write $U_{12}$ and $U_{123}$ for the induced composites $U_{1}\comp U_{2}$ and $U_{1}\comp U_{2}\comp U_{3}$, respectively.

\subsection{Internal pregroupoids}
A \defn{pregroupoid}~\cite{Kock-Pregroupoids, Johnstone:Herds, Janelidze-Pedicchio} in $\C$ is a span
\begin{equation*}
(d,c)=\vcenter{\xymatrix@!0@=3em{& D \ar[ld]_-{d} \ar[rd]^-{c} \\ D_0 & & D_0'}}
\end{equation*}
together with a structure of the form
\begin{equation}\label{pullbacks}
\vcenter{\xymatrix@=3em{
D\times_{D_{0}}D\times_{D'_{0}}D 
\ar@<.5ex>[r]^-{p_2}
\ar@<-.5ex>[d]_-{p_1}
\ar@{}[rd]|-{(1)}
& D\times_{D'_{0}}D \ar[r]^-{c_2} \ar[d]|-{c_1}
\ar@<0.5ex>[l]^-{i_2}
\ar@{}[rd]|-{(2)}
& D \ar[d]^-{c} \\
D\times_{D_{0}}D \ar[r]|-{d_2} \ar[d]_-{d_1}
\ar@<-0.5ex>[u]_-{i_1}
\ar@{}[rd]|-{(3)}
& D \ar[r]_-{c} \ar[d]^-{d} & D'_0 \\
D \ar[r]_-{d} & D_0
}}
\end{equation}
where (1), (2) and (3) are pullback squares, the morphisms $i_1$, $i_2$ are determined by
\[
p_1 i_1=1_{D\times_{D_{0}}D}, \quad p_2 i_1=\langle d_2 , d_2\rangle 
\]
and
\[
p_2 i_2=1_{D\times_{D'_{0}}D}, \quad p_1 i_2=\langle c_1 , c_1\rangle 
\]
and there is a further morphism $p\colon D\times_{D_{0}}D\times_{D'_{0}}D \to D$ which satisfies the conditions
\begin{gather}
p i_1 =d_1\quad\text{and}\quad p i_2 =c_2,\label{Mal'tsev-conditions}\\
dp =d c_2 p_2\quad\text{and}\quad cp = c d_1 p_1.\label{Domain-and-Codomain}
\end{gather}
When $\C$ admits split pullbacks and kernel pairs, we shall refer to a pregroupoid structure simply as a structure
\begin{equation}\label{pregroupoid}
\vcenter{\xymatrix@C=5em@R=2em{&& D'_0 \\
D\times_{D_{0}}D\times_{D'_{0}}D \ar[r]^-{p} & D \ar[ru]^-{c} \ar[rd]_-{d} \\
& & D_0.}}
\end{equation}
In order to have a visual picture, we may think of the object $D$ as having elements of the form
\[
\xymatrix@1{c(x) & \ar[l]_-{x} d(x)}\quad\text{or}\quad\xymatrix@1{\cdot & \ar[l]_-{x} \cdot}
\]
and hence the ``elements'' of $D\times_{D_{0}}D$, $D\times_{D'_{0}}D$ and $D\times_{D_{0}}D\times_{D'_{0}}D$ are, respectively, of the form
\[
\xymatrix@1{\cdot & \cdot \ar[l]_-{x} \ar[r]^-{y} & \cdot},\quad
\xymatrix@1{\cdot \ar[r]^-{x} & \cdot & \cdot \ar[l]_-{y}}
\]
and
\[
\xymatrix@1{\cdot & \ar[l]_-{x} \cdot \ar[r]^-{y} & \cdot & \ar[l]_-{z} \cdot.}
\]
Observe that the morphism $p$ is a kind of Mal'tsev operation in the
sense that $p(x,y,y)=x$ and $p(x,x,y)=y$ (the conditions~\eqref{Mal'tsev-conditions}). Furthermore, $dp(x,y,z)=dz$ and $cp(x,y,z)=cx$ by~\eqref{Domain-and-Codomain}.

In the following sections we shall also consider the forgetful functor
\[
V\colon \PreGrpd(\C) \to \Span(\C)
\]
from the category of pregroupoids to the category of spans in~$\C$.

The definition of pregroupoid also contains the associativity axiom, asking that $p(p(x,y,z),u,v)=p(x,y,p(z,u,v))$ whenever both sides of the equation make sense. We shall not assume this, but rather deduce the property in the naturally Mal'tsev and (weakly) Mal'tsev contexts.

\subsection{Relations}
The notions of reflexive relation, preorder (or reflexive and transitive relation), equivalence relation, and difunctional relation, may all be obtained, respectively, from the notions of reflexive graph, internal category (or multiplicative graph), internal groupoid, and pregroupoid, simply by imposing the extra condition that the pair of morphisms $(d,c)$ is jointly monomorphic.
We will also consider \defn{strong} relations: here the pair of morphisms $(d,c)$ is jointly strongly monomorphic.

\section{Mal'tsev conditions}\label{Section-Maltsev}
In this section we study some established and some less known characterisations of \emph{Mal'tsev} and \emph{naturally Mal'tsev} categories in terms of internal categorical structures. We extend these characterisations, which are usually considered in a context with finite limits, to a more general setting: categories with kernel pairs and split pullbacks. In particular we shall never assume that binary products exist. This allows for a treatment of \emph{weakly Mal'tsev} categories in a manner completely parallel to the treatment of the two stronger notions.

\subsection{Naturally Mal'tsev categories}
We first consider the notion of naturally Mal'tsev category~\cite{Johnstone:Maltsev} in a context where binary products are not assumed to exist. This may seem strange, as the original definition takes place in a category with binary products (and no other limits). We can do this because the main characterisation of naturally Mal'tsev categories---as those categories for which the forgetful functor from internal groupoids to reflexive graphs is an isomorphism---is generally stated in a finitely complete context. This context may be even further reduced: we shall show that the existence of kernel pairs and split pullbacks is sufficient.

\begin{theorem}\label{th1.nat}
Let $\C$ be a category with kernel pairs and split pullbacks. The following are equivalent:
\renewcommand{\labelenumi}{(\roman{enumi})}
\begin{enumerate}
\item the functor $U_{123}\colon \Grpd(\C) \to \RG(\C)$ is
an isomorphism;
\item the functor $U_{12}\colon \Cat(\C) \to \RG(\C)$ has a section;
\item the functor $U_{1}\colon \MG(\C) \to \RG(\C)$ has a section;
\item the functor $V\colon \PreGrpd(\C) \to \Span(\C)$ has a section;
\item for every diagram such as~\eqref{couniv} in $\C$, given any span
\begin{equation*}
\xymatrix@!0@=3em{& D \ar[ld]_-{d} \ar[rd]^-{c} \\ D_0 & & D_0'}
\end{equation*}
such that $d\alpha=d\beta f$ and $c\gamma=c\beta g$, there is a unique $\varphi \colon{A \times_B C \to D}$ such that
\begin{equation}\label{conditions}
\varphi e_1 = \alpha,\quad\varphi e_2=\gamma \quad\text{and}\quad d\varphi =d\gamma\pi_2,\quad c\varphi =c\alpha\pi_1.
\end{equation}
\end{enumerate}
If the above equivalent conditions hold, then the functors $U_{12}$, $U_{1}$ and $V$ are also isomorphisms. Furthermore, any pregroupoid is associative.
\end{theorem}
\begin{proof}
$\text{(i)}\Rightarrow\text{(ii)}$ follows by composing the inverse of $U_{123}$ from $\text{(i)}$ with the functor $U_{3}\colon{\Grpd(\C)\to \Cat(\C})$. For $\text{(ii)}\Rightarrow\text{(iii)}$ we compose with $U_{2}\colon{\Cat(\C)\to \MG(\C)}$. Let us prove $\text{(iii)}\Rightarrow\text{(iv)}$.

Suppose that the functor $U_1$ has a section. Then any reflexive graph admits a canonical morphism $m$
\begin{equation*}
\xymatrix@!0@=4em{ C_2 \ar[r]^-{m} & C_1 \ar@<1ex>[r]^-{d} \ar@<-1ex>[r]_-{c} &
C_0 \ar[l]|-{e} }
\end{equation*}
such that $m e_1=1_{C_{1}}=me_2$, $dm=d\pi_2$ and $cm=c\pi_1$ as in the definition of a multiplicative graph. Furthermore, this morphism is natural, in the sense that, for any
morphism $f=(f_1,f_0)$ of reflexive graphs, the diagram
\begin{equation}\label{m is natural}
\vcenter{\xymatrix@!0@=4em{C_2 \ar[r]^-{m} \ar[d]_-{f_2} & C_1 \ar[d]_-{f_1} \ar@<1ex>[r]^-{d} \ar@<-1ex>[r]_-{c} & C_0 \ar[l]|-{e} \ar[d]^-{f_{0}} \\
C'_2 \ar[r]_-{m'} & C'_1 \ar@<1ex>[r]^-{d'} \ar@<-1ex>[r]_-{c'} & C'_0 \ar[l]|-{e'}}}
\end{equation}
with $f_2=f_1\times_{f_0} f_1$ commutes.

To prove that the functor $V$ has a section, we have to construct a pregroupoid structure for any given span
\begin{equation*}
\xymatrix@!0@=3em{& D \ar[ld]_{d} \ar[rd]^{c} \\ D_0 & & D'_0.}
\end{equation*}
Let us consider the reflexive graph
\begin{equation}\label{rg1}
\xymatrix@=3em{D\times_{D_{0}}D\times_{D'_{0}}D \ar@<1ex>[r]^-{c_{2}p_{2}} \ar@<-1ex>[r]_-{d_{1}p_{1}} & D \ar[l]|-{\Delta}}
\end{equation}
(see diagram~\eqref{pullbacks}) where an ``element'' of $D\times_{D_{0}}D\times_{D'_{0}}D$
\begin{equation*}
\xymatrix@1{\cdot& \cdot \ar[l]_{x} \ar[r]^{y} & \cdot & \ar[l]_{z} \cdot}
\end{equation*}
is viewed as an arrow $y$ having domain $x$ and codomain $z$. It is clearly reflexive, with $\Delta(x)=(x,x,x)$ being the identity on $x$. It is a multiplicative graph because the functor~$U_{1}$ has a section. The desired pregroupoid structure $p$ for $(D,d,c)$ is obtained by the following procedure: given
\begin{equation*}
\xymatrix@1{\cdot& \cdot \ar[l]_{x} \ar[r]^{y} & \cdot & \ar[l]_{z}\cdot}
\end{equation*}
in $D\times_{D_{0}}D\times_{D'_{0}}D$, consider the pair of composable arrows
\begin{equation*}
(\xymatrix@1{\cdot\ar[r]^{x} & \cdot & \cdot \ar[l]_{x} \ar[r]^{y} &\cdot}, \xymatrix@1{\cdot\ar[r]^{y} & \cdot & \cdot \ar[l]_{z} \ar[r]^{z}
&\cdot} )
\end{equation*}
in the reflexive graph~\eqref{rg1}. Since this reflexive graph is multiplicative, multiply in order to obtain
\begin{equation*}
\xymatrix@1{\cdot\ar[r]^{x} & \cdot & \cdot \ar[l]_{p(x,y,z)} \ar[r]^{z}
&\cdot}
\end{equation*}
and project to the middle component.

The equalities $p(x,y,y)=x$ and $p(x,x,y)=y$ simply follow from the multiplicative identities $me_1=1_{C_{1}}=me_2$ of the multiplicative graph. Likewise, $dp(x,y,z)=dz$ and $cp(x,y,z)=cx$. This construction is functorial because the multiplication is natural. 

Next we prove that, if $V$ has a section, then the category $\C$ satisfies Condition~(v). Consider a diagram such as~\eqref{kite} above and a suitable span $(d,c)$. We have to construct a morphism $\varphi\colon {A\times_{B}C \to D}$ which satisfies the needed conditions, and prove that this $\varphi$ is unique. To do so, we use the natural pregroupoid structure $p\colon D\times_{D_{0}}D\times_{D'_{0}}D \to D$. Since $d\alpha=d\beta f$, $c\gamma=c\beta g$ and $\alpha r=\beta=\gamma s$, there is an induced morphism
\[
\langle \alpha \pi_{1},\beta f\pi_{1},\gamma\pi_{2}\rangle \colon A\times_B C \to D\times_{D_{0}}D\times_{D'_{0}}D.
\]
It assigns to any $(a,c)$ with $f(a)=b=g(c)$ in $A\times_B C$ a triple
\begin{equation*}
\xymatrix@1{\cdot & \cdot \ar[l]_-{\alpha(a)} \ar[r]^-{\beta(b)} &
\cdot & \cdot \ar[l]_-{\gamma(c)}}
\end{equation*}
in $D\times_{D_{0}}D\times_{D'_{0}}D$. The desired morphism $\varphi \colon A\times_{B} C \to D$
is then obtained by taking its composition in the pregroupoid, i.e., $\varphi (a,c) =p(\alpha(a),\beta(b),\gamma(c))$ or
\[
\varphi =p\langle \alpha \pi_{1},\beta f\pi_{1},\gamma\pi_{2}\rangle.
\]
This proves existence; the equalities $\varphi(a,b,s(b))=\alpha(a)$ and
$\varphi(r(b),b,c)=\gamma(c)$ follow from the properties of $p$, as do $d\varphi =d\gamma\pi_2$ and $c\varphi =c\alpha\pi_1.$

Now we show that the equalities~\eqref{conditions} determine $\varphi$ uniquely. Let us consider the span
\[
\xymatrix@!0@=3em{& A\times_{B}C \ar[ld]_{\pi_{2}} \ar[rd]^{\pi_{1}} \\ C & & A}
\]
with its induced pregroupoid structure
\[
p\colon (A\times_{B}C)\times_{C}(A\times_{B}C)\times_{A}(A\times_{B}C)\to A\times_{B}C;
\]
if the morphisms in this pregroupoid are viewed as arrows
\[
\xymatrix@1{a & c \ar[l]_-{(a,c)}}
\]
then the operation $p$ takes a composable triple
\[
\xymatrix@1{a & c \ar[l]_-{(a,c)} \ar[r]^-{(a',c)} & a' & c' \ar[l]_-{(a',c')}}
\]
and sends it to
\[
\xymatrix@1{a & c' \ar[l]_-{(a,c')}}
\]
in $A\times_{B}C$. Note that this pregroupoid structure is unique, because the given span is a relation; in fact, its existence expresses the relation's difunctionality. Further note that it is a strong relation (cf.\ Theorem~\ref{th6.weak} below).

The morphism $\varphi$ now gives rise to a morphism of pregroupoids, determined by the morphism of spans
\[
\xymatrix@!0@=4em{C \ar[d]_-{d\gamma} & A\times_{B}C \ar[d]^-{\varphi} \ar[r]^-{\pi_{1}} \ar[l]_-{\pi_{2}} & A \ar[d]^-{c\alpha} \\
D_{0} & D \ar[r]_-{c} \ar[l]^-{d} & D_{0}'.}
\]
We write
\[
\varphi'\colon{(A\times_{B}C)\times_{C}(A\times_{B}C)\times_{A}(A\times_{B}C)\to D\times_{D_{0}}D\times_{D'_{0}}D}
\]
for the induced morphism to see that
\begin{align*}
\varphi(a,c) &= \varphi p(\xymatrix@1{a & sf(a) \ar[l] \ar[r] & rg(c) & c \ar[l]})\\
&= p\varphi'(\xymatrix@1{a & sf(a) \ar[l] \ar[r] & rg(c) & c \ar[l]})\\
&=p(\xymatrix@1@=3em{\cdot & \cdot \ar[l]_-{\varphi e_{1}(a)} \ar[r]^-{\varphi e_{1}r(b)} &
\cdot & \cdot \ar[l]_-{\varphi e_{2}(c)}})\\
&=p(\xymatrix@1{\cdot & \cdot \ar[l]_-{\alpha(a)} \ar[r]^-{\beta(b)} &
\cdot & \cdot \ar[l]_-{\gamma(c)}})\\
&=p(\alpha(a),\beta(b),\gamma(c))
\end{align*}
and $\varphi$ is uniquely determined.

Next we prove that (v) implies Condition (i) in our theorem. Given a reflexive graph~\eqref{reflgraph}, a unique multiplication $m$ satisfying~\eqref{conditions}, so
\[
m e_1 = 1_{C_{1}},\qquad m e_2=1_{C_{1}} \qquad\text{and}\qquad dm =d\pi_2,\qquad cm =c\pi_1,
\]
is induced by the diagram
\[
\vcenter{\xymatrix@!0@=4em{C_{1} \ar@<.5ex>[r]^-{d} \ar@{=}[rd] & C_{0}
\ar@<.5ex>[l]^-{e}
\ar@<-.5ex>[r]_-{e}
\ar[d]^-{e} & C_{1} \ar@<-.5ex>[l]_-{c} \ar@{=}[ld]\\
& C_{1}}}
\]
together with the span $(d,c)$. 

The naturality of $m$ (see diagram~\eqref{m is natural}) follows from the uniqueness of the morphism induced by the diagram
\[
\vcenter{\xymatrix@!0@=4em{C_{1} \ar@<.5ex>[r]^-{d} \ar@{->}[rd]_{f_1} & C_{0}
\ar@<.5ex>[l]^-{e}
\ar@<-.5ex>[r]_-{e}
\ar[d]|-{e'f_0} & C_{1} \ar@<-.5ex>[l]_-{c} \ar@{->}[ld]^{f_1}\\
& C'_{1}}}
\]
and the span $(d',c')$: indeed, both $f_{1}m$ and $m'f_{2}$ qualify. 
This already gives us Condition~(iii) in its strong form where $U_{1}$ is an isomorphism.

The associativity condition (needed for (ii)) follows from the uniqueness of the morphism induced by the diagram
\[
\vcenter{\xymatrix@!0@=4em{C_{2} \ar@<.5ex>[r]^-{\pi_{2}} \ar[rd]_-{m} & C_{1}
\ar@<.5ex>[l]^-{e_{2}}
\ar@<-.5ex>[r]_-{e_{1}}
\ar@{=}[d] & C_{2}: \ar@<-.5ex>[l]_-{\pi_{1}} \ar[ld]^-{m}\\
& C_{1}}}
\]
indeed, both $m(1_{C_{1}}\times m)$ and $m(m\times 1_{C_{1}})$ satisfy the required conditions~\eqref{conditions}, so they coincide.

The existence of inverses (needed for (i)) follows from the diagram
\[
\vcenter{\xymatrix@!0@=4em{C_{2} \ar@<.5ex>[r]^-{m} \ar[rd]_-{\pi_{2}} & C_{1}
\ar@<.5ex>[l]^-{e_{2}}
\ar@<-.5ex>[r]_-{e_{1}}
\ar@{=}[d] & C_{2} \ar@<-.5ex>[l]_-{m} \ar[ld]^-{\pi_{1}}\\
& C_{1}}}
\]
as explained in~\cite{NMF1}.

To show that the functor $V$ is an isomorphism, given a span $(d,c)$, we use the diagram
\[
\vcenter{\xymatrix@!0@=4em{D_{d,c} \ar@<.5ex>[r]^-{c_{2}p_2} \ar@{=}[rd] & D
\ar@<.5ex>[l]^-{\Delta}
\ar@<-.5ex>[r]_-{\Delta}
\ar[d]|-{\Delta} & D_{d,c} \ar@<-.5ex>[l]_-{d_1p_1} \ar@{=}[ld]\\
& D_{d,c}}}
\]
where $D_{d,c}=D\times_{D_{0}}D\times_{D_{0}'}D$ and $\Delta=\langle1_{D},1_{D},1_{D}\rangle$ to prove uniqueness of its pregroupoid structure.

Finally, given a pregroupoid~\eqref{pregroupoid}, its associativity follows by using (v) on the diagram
\[
\vcenter{\xymatrix@!0@=4em{D_{d,c} \ar@<.5ex>[r]^-{c_{2}p_2} \ar[rd]_-{p} & D
\ar@<.5ex>[l]^-{\Delta}
\ar@<-.5ex>[r]_-{\Delta}
\ar@{=}[d] & D_{d,c} \ar@<-.5ex>[l]_-{d_1p_1} \ar[ld]^-{p}\\
& D,}}
\]
because the morphisms
\[
{D_{d,c}\times_{D} D_{d,c}\to D}
\]
defined by sending $(x,y,z,u,v)$ to $p(p(x,y,z),u,v)$ or to $p(x,y,p(z,u,v))$ both meet the requirements, so they must agree by the uniqueness in (v).
\end{proof}

Observe that, in the case of finite limits, any one of the equivalent conditions of Theorem~\ref{th1.nat} is a characterisation for the notion of naturally Mal'tsev category introduced in~\cite{Johnstone:Maltsev}. Indeed, the Mal'tsev operation on an object $X$ is determined by the diagram
\[
\vcenter{\xymatrix@!0@=5em{X\times X \ar@<.5ex>[r]^-{\pi_{2}} \ar[rd]_-{\pi_{1}} & X
\ar@<.5ex>[l]^-{\langle 1_{X},1_{X}\rangle}
\ar@<-.5ex>[r]_-{\langle 1_{X},1_{X}\rangle}
\ar@{=}[d] & X\times X \ar@<-.5ex>[l]_-{\pi_{1}} \ar[ld]^-{\pi_{2}}\\
& X}}
\]
together with the span ${1\ot X\to 1}$.

In the presence of coequalisers, when every span in $\C$ is naturally endowed with a unique pregroupoid structure, there is an interchange law for composable strings valid in any pregroupoid in~$\C$.

\begin{proposition}\label{proposition-interchange}
Let $\C$ be a category with kernel pairs, split pullbacks and coequalisers satisfying the conditions {\rm (i)--(v)}. Consider a pregroupoid~\eqref{pregroupoid} in~$\C$. Then for any configuration of the shape
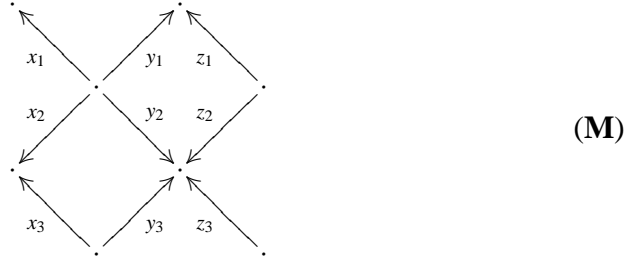
\begin{equation}\label{config}
\vcenter{\matrixone}
\end{equation}
in this pregroupoid, the equality
\begin{multline}\label{p_autonomous}
p(p(x_1,x_2,x_3),p(y_1,y_2,y_3),p(z_1,z_2,z_3))\\
=p(p(x_1,y_1,z_1),p(x_2,y_2,z_2),p(x_3,y_3,z_3))
\end{multline}
holds.
\end{proposition}
\begin{proof}
It suffices to consider the pregroupoid in $\C$ in which the configurations~\eqref{config} are the composable triples, and then the equality will follow by naturality of the pregroupoid structures. This pregroupoid 
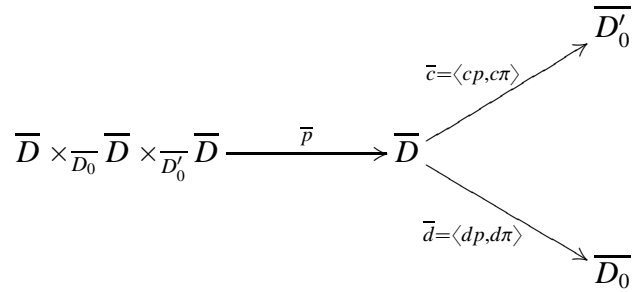
\[
\vcenter{\xymatrix@C=5em@R=2em{&& \overline{D'_{0}} \\
\overline{D}\times_{\overline{D_{0}}}\overline{D}\times_{\overline{D'_{0}}}\overline{D} \ar[r]^-{\overline{p}} & \overline{D} \ar[rd]_-{\overline{d}=\langle dp, d\pi\rangle} \ar[ru]^-{\overline{c}=\langle cp,c\pi\rangle} \\
& & \overline{D_{0}}}}
\]
is determined by the span $(\langle dp, d\pi\rangle,\langle cp,c\pi\rangle)$ where $\overline{D}=D\times_{D_{0}}D\times_{D'_{0}}D$,
\begin{gather*}
\overline{D_{0}}=D_{0}\times_{Q}D_{0}\qquad\qquad \xymatrix@1{D_{0} \ar[rr]^-{\coeq(dp, d\pi)} && Q}\\
\overline{D'_{0}}=D'_{0}\times_{Q'}D'_{0}\qquad\qquad \xymatrix@1{D'_{0} \ar[rr]^-{\coeq(cp, c\pi)} && Q'}
\end{gather*}
and the middle projection $\pi=d_{2}p_{1}=c_{1}p_{2}\colon{D\times_{D_{0}}D\times_{D'_{0}}D\to D}$ (diagram~\eqref{pullbacks}) maps a composable triple $(x_{1},x_{2},x_{3})$ to $x_{2}$. It is easily checked that the morphism~$\overline{p}$ which sends~\eqref{config} to its horizontal composite---the composable triple
\[
(p(x_1,y_1,z_1),p(x_2,y_2,z_2),p(x_3,y_3,z_3))
\]
in $D$, see Figure~\ref{bigdiagram}---determines a pregroupoid structure (hence, the unique one) on this span. 

\begin{figure}
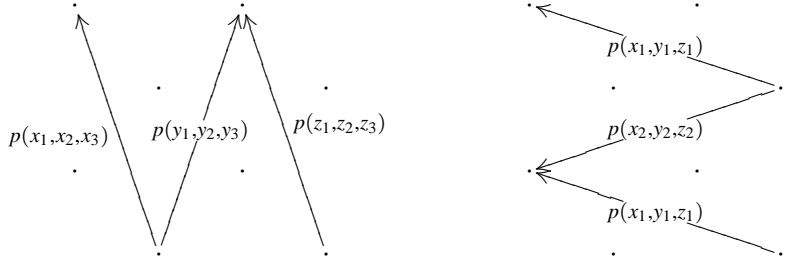

{$\matrixthree \qquad\qquad \matrixtwo$}
\caption{Vertical and horizontal composition}\label{bigdiagram}
\end{figure}

Furthermore, by naturality of pregroupoid structures, the morphism of spans
\[
\xymatrix@!0@=4.5em{\overline{D_{0}} \ar[d] & \overline{D} \ar[d]^-{p} \ar[r]^-{\langle cp,c\pi\rangle} \ar[l]_-{\langle dp, d\pi\rangle} & \overline{D'_{0}} \ar[d] \\
D_{0} & D \ar[r]_-{c} \ar[l]^-{d} & D_{0}'}
\]
induces a morphism $p'\colon{\overline{D}\times_{\overline{D_{0}}}\overline{D}\times_{\overline{D'_{0}}}\overline{D}\to \overline{D}}$ such that $pp'=p\overline{p}$, which gives us the required equality~\eqref{p_autonomous}. Indeed, the induced morphism $p'$ takes \eqref{config} and sends it to its vertical composite---the composable triple
\[
(p(x_1,x_2,x_3),p(y_1,y_2,y_3),p(z_1,z_2,z_3))
\]
in $D$, see again Figure~\ref{bigdiagram}.
\end{proof}

Note that the equality~\eqref{p_autonomous} is a partial version of the Mal'tsev operation $p$ being \emph{autonomous}, see~\cite{Johnstone:Maltsev}.

\subsection{Mal'tsev categories}
Restricting Theorem~\ref{th1.nat} to the case where the morphisms $d$ and $c$ are jointly monomorphic we obtain the well known characterisation~\cite{Carboni-Pedicchio-Pirovano} for Mal'tsev categories.

\begin{theorem}\label{th4.mal}
Let $\C$ be a category with kernel pairs and split pullbacks. The
following are equivalent:
\begin{enumerate}
\item [(i')] every reflexive relation is an equivalence relation;
\item [(ii')] every reflexive relation is a preorder;
\item [(iii')] every reflexive relation is transitive;
\item [(iv')] every relation is difunctional;
\item [(v')] for every diagram such as~\eqref{couniv} in $\C$, given any relation
\begin{equation*}
\xymatrix@!0@=3em{& D \ar[ld]_-{d} \ar[rd]^-{c} \\ D_0 & & D_0'}
\end{equation*}
such that $d\alpha=d\beta f$ and $c\gamma=c\beta g$, there is a unique $\varphi \colon{A \times_B C \to D}$ such that
\[
\varphi e_1 = \alpha,\qquad\varphi e_2=\gamma\qquad\text{and}\qquad d\varphi =d\gamma\pi_2,\qquad c\varphi =c\alpha\pi_1.
\]
\end{enumerate}
\end{theorem}
\begin{proof}
By restricting to relations one easily adapts the proof of Theorem~\ref{th1.nat} to the present situation.
\end{proof}

An important result on Mal'tsev categories is the following one, usually stated for finite limits~\cite{Carboni-Pedicchio-Pirovano}; it follows, for instance, from Theorem~\ref{th8.main}.

\begin{theorem}\label{th5.mal}
Let $\C$ be a category with kernel pairs, split pullbacks and equalisers, satisfying the equivalent conditions of Theorem~\ref{th4.mal}. Then the forgetful functor
\[
U_{3}\colon \Grpd(\C) \to \Cat(\C)
\]
is an isomorphism.\noproof
\end{theorem}

\subsection{Weakly Mal'tsev categories}
A category is said to be \defn{weakly Mal'tsev} when it has split pullbacks and every induced pair of morphisms into the pullback $(e_1,e_2)$ as in Diagram~\eqref{splitsquare} above is jointly epimorphic~\cite{NMF1}.

Further restricting the conditions of Theorem~\ref{th1.nat} to the case where the morphisms $d$ and $c$ are jointly strongly monomorphic---and calling such a span a \defn{strong relation}~\cite{ZurabNelson}---we obtain a characterisation of weakly Mal'tsev categories.

\begin{theorem}\label{th6.weak}
Let $\C$ be a category with kernel pairs and split pullbacks. The
following are equivalent:
\begin{enumerate}
\item [(i'')] every reflexive strong relation is an equivalence relation;
\item [(ii'')] every reflexive strong relation is a preorder;
\item [(iii'')] every reflexive strong relation is transitive;
\item [(iv'')] every strong relation is difunctional;
\item [(v'')] for every diagram such as~\eqref{couniv} in $\C$, given any strong relation
\begin{equation*}
\xymatrix@!0@=3em{& D \ar[ld]_-{d} \ar[rd]^-{c} \\ D_0 & & D_0'}
\end{equation*}
such that $d\alpha=d\beta f$ and $c\gamma=c\beta g$, there is a unique $\varphi \colon{A \times_B C \to D}$ such that
\[
\varphi e_1 = \alpha,\qquad\varphi e_2=\gamma\qquad\text{and}\qquad d\varphi =d\gamma\pi_2,\qquad c\varphi =c\alpha\pi_1.
\]
\end{enumerate}
\end{theorem}
\begin{proof}
By restricting to strong relations one easily adapts the proof of Theorem~\ref{th1.nat} to the present situation.
\end{proof}

\begin{theorem}\label{th7.weak}
Let $\C$ be a category with kernel pairs, split pullbacks and equalisers. The following are equivalent:
\begin{enumerate}
\item $\C$ is a weakly Mal'tsev category;
\item $\C$ satisfies the equivalent conditions of Theorem~\ref{th6.weak}.
\end{enumerate}
\end{theorem}
\begin{proof}
In the presence of equalisers, the weak Mal'tsev axiom is equivalent to Condition~(iv'')---see~\cite{ZurabNelson}.
\end{proof}

Mimicking the argument at the end of the proof of Theorem~\ref{th1.nat}, it is easily seen that in a weakly Mal'tsev category, any internal pregroupoid is associative. The corresponding result for internal multiplicative graphs is treated in the following section.

\section{Internal categories vs.\ internal groupoids}\label{Section-Cat-vs-Grpd}
We prove that, in a weakly Mal'tsev category with kernel pairs and equalisers, internal categories are internal groupoids if and only if every preorder is an equivalence relation.

\begin{theorem}\label{th8.main}
Let $\C$ be a weakly Mal'tsev category with kernel pairs and equalisers. Then:
\begin{enumerate}
\item the forgetful functor
\[
U_{2}\colon\Cat(\C)\to \MG(\C)
\]
is an isomorphism;
\item the forgetful functor
\[
U_{3}\colon\Grpd(\C) \to\Cat(\C)
\]
is an isomorphism if and only if every internal preorder in $\C$ is an equivalence relation.
\end{enumerate}
\end{theorem}

Part (1) of this result was already obtained in~\cite{NMF1} where the definition of multiplicative graph does not include the conditions $dm=d\pi_2$ and $cm=c\pi_1$. Indeed, in this context they automatically hold. The proof of Part~(2) depends on the following lemma.

\begin{lemma}\label{Lemma-Relation}
Let $\C$ be a weakly Mal'tsev category with equalisers. Given a category~\eqref{multgraph} in $\C$, the morphisms
\begin{equation*}
\langle \pi_1,m\rangle\colon{C_{2}\to C_{1}\times_{c}C_{1}}\quad\text{and}\quad \langle m,\pi_2\rangle\colon C_{2}\to C_{1}\times_{d}C_{1}
\end{equation*}
are monomorphisms; this means that the multiplication is cancellable
on both sides.
\end{lemma}
\begin{proof}
We shall prove $\langle \pi_1,m\rangle $ is a monomorphism. A similar argument
shows the same for $\langle m,\pi_2\rangle $.

First observe that the kernel pairs $C_1\times_cC_1$, $C_1\times_dC_1$, $C_2\times_mC_2$, $C_2\times_{\pi_1}C_2$ and $C_2\times_{\pi_2}C_2$ exist because $c$, $d$, $m$, $\pi_1$ and $\pi_2$ are split epimorphisms. To prove that $\langle \pi_1,m\rangle $ is a monomorphism is the same as proving for every $x$, $y\colon {Z \to C_2}$ that
\begin{equation*}
\left.\begin{aligned}
&\pi_1 x=\pi_1 y\\
&m x=m y\end{aligned} \right\}
\quad\Rightarrow\quad 
\pi_2 x = \pi_2 y.
\end{equation*}
Assuming that $\pi_1 x=\pi_1 y$ we have induced morphisms
\[
\langle x,y\rangle
\quad\text{and}\quad
\langle e_2\pi_2 x,e_2 \pi_2 y\rangle\colon Z\to C_2\times_{\pi_1}C_2.
\]
Indeed, $\pi_1 e_2 \pi_2 x = \pi_1 e_2 \pi_2 y$ as $\pi_1 e_2 \pi_2 = ec \pi_2 =ed\pi_1$. Considering the equaliser $(S,\langle s_1,s_2\rangle)$ of the pair of morphisms
\[
\xymatrix@!0@=5em{C_2\times_{\pi_1}C_2 \ar@<.5ex>[r] \ar@<-.5ex>[r] & C_2 \ar[r]^{m} &
C_1,}
\]
and identifying $C_2\times_{\pi_1}C_2$ with $C_1\times_{C_0}(C_1\times_cC_1)$ we obtain a strong relation
\begin{equation*}
\xymatrix@!0@=3em{& S \ar[dl]_{s_1} \ar[rd]^{s_2} \\ C_1 & & C_1\times_{c}C_1}
\end{equation*}
which may be pictured as
\begin{equation*}
\xymatrix@1{ \cdot \ar@{=}[d] & \cdot \ar@{=}[d] \ar[l]_-{x_1} & \cdot
\ar[l]_-{x_2} \\
\cdot & \cdot \ar[l]^-{y_1} & \cdot \ar[l]^-{y_2}
}
\end{equation*}
with $x_1=y_1$ and $(x_{1}=y_{1})S(x_{2},y_{2})$ if and only if $x_1x_2=y_1y_2$.

By Theorem~\ref{th6.weak}, this relation, being a strong relation,
is also difunctional and the argument used on page 103 of~\cite{Carboni-Pedicchio-Pirovano} also applies
here to show that
\[
\langle e_2\pi_2x,e_2\pi_2y\rangle =\langle s_1,s_2\rangle pi \overline{\langle x,y\rangle },
\]
where $p\colon{SS^{-1}S\to S}$ is obtained by difunctionality, $\overline{\langle x,y\rangle }\colon{Z\to S}$ is the factorisation of $\langle x,y\rangle $ through the equaliser (we are assuming that $mx=my$), and the morphism $i\colon{S\to SS^{-1}S}$, which sends $(x_{1}=y_{1})S(x_{2},y_{2})$ to \[(1=1)S(1,1)S^{-1}(x_{1}=y_{1})S(x_{2},y_{2}),\] may be pictured as follows.
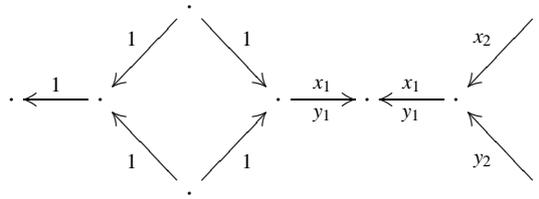
\[\xymatrix{&&\cdot \ar[ld]_{1}\ar[dr]^{1}&&&& \cdot \ar[ld]_{x_2} \\ \cdot & \cdot \ar[l]_{1}&& \cdot \ar[r]_{y_1}^{x_1} & \cdot & \cdot \ar[l]_{x_1}^{y_1}\\
&&\cdot \ar[lu]^{1} \ar[ru]_{1}&&&& \cdot\ar[lu]^ {y_2} 
}\]
This proves that $\langle e_2\pi_2x,e_2\pi_2y\rangle $ factors through the equaliser $S$, so we may conclude that
\[
m e_2\pi_2x=m e_2\pi_2y,
\]
or $\pi_2 x= \pi_2 y$ as desired.\end{proof}
 
\begin{proof}[Proof of Theorem~\ref{th8.main}]
If the functor $U_{3}$ is an isomorphism then in particular any preorder is an equivalence relation. For the converse, assume that every preorder is an equivalence relation (and every strong relation is difunctional). Given any category~\eqref{multgraph} we shall prove that it is a groupoid. For this to happen it suffices that there is a morphism $t\colon C_1\to C_1$ with $ct=d$ and $m\langle 1_{C_{1}},t\rangle =ec$ (see, for instance,~\cite{NMF1}).

By Lemma~\ref{Lemma-Relation} we already know that the morphisms
$\langle m,\pi_2\rangle$ and $\langle \pi_1,m\rangle $ are monomorphisms. This means that the
reflexive graph
\begin{equation*}
\xymatrix@!0@=4em{ C_2 \ar@<1ex>[r]^-{m} \ar@<-1ex>[r]_-{\pi_1} & C_1
\ar[l]|-{e_1}}
\end{equation*}
is a reflexive relation, and since it is transitive---by assumption it is a multiplicative graph---it is an equivalence relation. Hence there is a morphism
\begin{equation*}
\tau=\langle m,q\rangle\colon{C_2 \to C_2}
\end{equation*}
such that $m\tau = \pi_1$. Now $t=q e_2$ is the needed morphism ${C_{1}\to C_{1}}$. Indeed $dm=cq$, because $\langle m,q\rangle $ is a morphism into the pullback
$C_2$, so that 
\[
ct=cqe_2=dme_2=d;
\]
furthermore,
\[
m\langle 1_{C_{1}},t\rangle =m\langle me_2,qe_2\rangle =m\langle m,q\rangle e_2=\pi_{1} e_2 = ec,
\]
which completes the proof.
\end{proof}

\begin{remark}
In general, a category can be weakly Mal'tsev without Condition~(2) of Theorem~\ref{th8.main} holding. For instance, in the category of commutative monoids with cancellation, the relation $\leq$ on the monoid of natural numbers~$\N$ is a preorder which is not an equivalence relation.
\end{remark}

\begin{remark}
It is possible for a category to satisfy both Condition~(1) and Condition~(2) of Theorem~\ref{th8.main} without being Mal'tsev: see the following section.
\end{remark}

\section{The varietal case}\label{Section-Varieties}
When we restrict to varieties, the condition ``every internal preorder is an equivalence relation'' singled out in part (2) of Theorem~\ref{th8.main} is known to be equivalent to the variety being $n$-permutable for some $n$. We explain how to prove this when passing via a characterisation of $n$-permutability due to Hagemann.

\subsection{Finitary quasivarieties} 
Just like a variety of algebras is determined by certain identities between terms, a quasivariety also admits quasi-identities in its definition, i.e., expressions of the form
\[
\resizebox{\textwidth}{!}
{\mbox{$
\left.\begin{aligned}
v_{1}(x_{1},\dots, x_{k})&=w_{1}(x_{1},\dots, x_{k})\\
&\;\;\vdots\\
v_{n}(x_{1},\dots, x_{k})&=w_{n}(x_{1},\dots, x_{k})\end{aligned} \right\}
\quad\Rightarrow\quad v_{n+1}(x_{1},\dots, x_{k})=w_{n+1}(x_{1},\dots, x_{k})
$}}
\]
---see, for instance, \cite{Maltsev} for more details. It is well known that any quasivariety may be obtained as a regular epi-reflective subcategory of a variety, and more generally the sub-quasivarieties of a quasivariety correspond to its regular epi-reflective subcategories. In particular, sub-quasivarieties are closed under subobjects.

\subsection{$n$-Permutable varieties}
The following equivalent conditions due to Hagemann~\cite{Hagemann-Mitschke} describe what it means for a variety to be \defn{$n$-permutable}. (Recall that $2$-permutability is just the Mal'tsev property and a regular category which is $3$-permutable is called \defn{Goursat}~\cite{Carboni-Kelly-Pedicchio}.)

\begin{proposition}\label{Proposition-n-Permutable}
For a finitary quasivariety $\V$ and a natural number $n\geq 2$, the following are equivalent:
\begin{enumerate}
\item for any two equivalence relations $R$ and $S$ on an object $A$, we have $(R,S)_{n}=(S,R)_{n}$;
\item there exist $n-1$ terms $w_{1}$, \dots, $w_{n-1}$ in $\V$ such that
\[
\left\{\begin{aligned}
&w_{1}(x,z,z)=x\\
&w_{i}(x,x,z)=w_{i+1}(x,z,z)\\
&w_{n-1}(x,x,z)=z;
\end{aligned}\right.
\]
\item for any reflexive relation~$R$, we have $R^{-1}\subset R^{n-1}$.
\end{enumerate}
\end{proposition}

In fact, this result is valid in regular categories, as shown in~\cite{JRVdL1}. Also the following result is known~\cite{CR}:

\begin{proposition}\label{Proposition-Permutable}
For a finitary quasivariety $\V$, the following are equivalent:
\begin{enumerate}
\item in $\V$, every internal preorder is an equivalence relation;
\item $\V$ is $n$-permutable for some $n$.
\end{enumerate}
\end{proposition}
\begin{proof}
By Proposition~\ref{Proposition-n-Permutable}, if Condition (2) holds then for every reflexive relation~$R$ in $\V$ we have that $R^{-1}\subset R^{n-1}$. Now if $R$ is transitive then ${R^{n-1}\subset R}$, so that $R^{-1}\subset R$, which means that $R$ is symmetric.

To prove the converse, suppose that every internal preorder in $\V$ is an equivalence relation. Let $A$ be the free algebra on the set $\{x,z\}$ and let~$R$ be the reflexive relation on $A$ consisting of all pairs
\[
(w(x,x,z),w(x,z,z))
\]
for $w$ a ternary term. Then the pair $(x,z)$ is in~$R$. By assumption, the transitive closure~$\overline{R}$ of~$R$ is also symmetric, hence contains the pair $(z,x)$. This means that $(z,x)$ may be expressed through a chain of finite length in~$R$. More precisely, there exists a natural number $n$ and ternary terms $w_{1}$, \dots, $w_{n-1}$ such that
\begin{multline*}
z=w_{n-1}(x,x,z)Rw_{n-1}(x,z,z)=w_{n-2}(x,x,z)Rw_{n-2}(z,z,x)=\\
\dots=w_{1}(x,x,z)Rw_{1}(x,z,z)=x.
\end{multline*}
By Proposition~\ref{Proposition-n-Permutable} this means that $\V$ is $n$-permutable.
\end{proof}

\begin{remark}
This of course raises the question whether a similar result would hold in a purely categorical context. It seems difficult to obtain the number $n$ which occurs in Condition (2) of Proposition~\ref{Proposition-Permutable} without using free algebra structures, which are not available in general. And indeed, a counterexample exists~\cite{MFRVdL1}. On the other hand, the implication (2) \implies\ (1) admits a proof which is \emph{almost} categorical---but depends on a characterisation of $n$-permutability for regular categories as in Condition~(3) of Proposition~\ref{Proposition-n-Permutable}. This is the subject of the articles~\cite{RVdL4} and~\cite{JRVdL1}.
\end{remark}

\begin{remark}
Through Theorem~\ref{th8.main}, this result implies that in an $n$-permu\-table weakly Mal'tsev variety, every internal category is an internal groupoid. On the other hand, using different techniques, and without assuming the weak Mal'tsev condition, Rodelo recently proved that in any $n$-permutable variety, internal categories and internal groupoids coincide~\cite{Rodelo:Internal-categories}. Whence the question: how different are $n$-permutable varieties from weakly Mal'tsev ones? The only thing we know about this so far is that the two conditions together are not strong enough to imply that the variety is Mal'tsev (see Example~\ref{All-but-Maltsev}). Further note that the conditions (IC1) and (IC2) considered in the paper~\cite{Rodelo:Internal-categories}, that is, $dm =d\pi_2$ and $cm = c\pi_1$ in~\eqref{multgraph}, come for free in a weakly Mal'tsev category. Outside this context, however, it is no longer clear whether or not they will always hold.
\end{remark}

\subsection{Constructing weakly Mal'tsev quasivarieties}

A $3$-permutable (qua\-si)va\-ri\-ety always contains a canonical subvariety which is also weakly Mal'tsev. This allows us to construct examples of weakly Mal'tsev categories which are $3$-permutable but not $2$-permutable---thus we see, in particular, that in a weakly Mal'tsev category~$\C$, categories and groupoids may coincide, even without~$\C$ being Mal'tsev.

\begin{proposition}\label{Proposition-Subvariety}
Let $\V$ be a Goursat finitary quasivariety with $w_{1}$, $w_{2}$ the terms obtained using Proposition~\ref{Proposition-n-Permutable}. Then the sub-quasivariety $\W$ of~$\V$ defined by the quasi-identity
\[
\left.\begin{aligned}
&w_{1}(x,a,b)=w_{2}(a,b,c)=w_{1}(x',a,b)\\
&w_{2}(b,c,x)=w_{1}(a,b,c)=w_{2}(b,c,x')\end{aligned} \right\}
\quad\Rightarrow\quad x=x'
\]
is weakly Mal'tsev.
\end{proposition}
\begin{proof}
For any split pullback
\[
\vcenter{\xymatrix@!0@=4em{
A \times_{B} C \ar@<.5ex>[r]^-{p_2} \ar@<-.5ex>[d]_-{p_1} & C \ar@<.5ex>[l]^-{e_2}
\ar@<-.5ex>[d]_-{g}
 \\
A
 \ar@<.5ex>[r]^-{f} \ar@<-.5ex>[u]_-{e_1}
& B
 \ar@<.5ex>[l]^-{r} \ar@<-.5ex>[u]_-{s}
 }}
\]
we have to show that $e_{1}$ and $e_{2}$ are jointly epic: any two $\varphi$, $\varphi'\colon{A \times_{B} C \to D}$ such that
\[
\varphi e_1 = \alpha = \varphi ' e_1\quad\text{and}\quad\varphi e_2 = \gamma = \varphi ' e_2
\]
must coincide. We use the notations from Diagram~\eqref{kite} and consider $a\in A$ and $c\in C$ with $f(a)=b=g(c)$. Then
\begin{align*}
w_{1}(\varphi(a,c),\alpha(a),\beta(b)) & = w_{1}(\varphi(a,c),\varphi(a,s(b)),\varphi(r(b),s(b))) \\
 & =\varphi(w_{1}(a,a,r(b)),w_{1}(c,s(b),s(b)))\\
 & =\varphi(w_{2}(a, r(b),r(b)),c)\\
 & =\varphi(w_{2}(a, r(b),r(b)),w_{2}(s(b),s(b),c))\\
 & =w_{2}(\varphi(a, s(b)),\varphi(r(b),s(b)),\varphi(r(b),c))\\
 & =w_{2}(\alpha(a),\beta(b),\gamma(c))
\end{align*}
and
\begin{align*}
w_{2}(\beta(b), \gamma(c), \varphi(a,c)) & =w_{2}(\varphi(r(b),s(b)),\varphi(r(b),c),\varphi(a,c)) \\
 & =\varphi(w_{2}(r(b),r(b),a),w_{2}(s(b),c,c))\\
 & =\varphi(a,w_{1}(s(b),s(b),c)))\\
 & =\varphi(w_{1}(a,r(b),r(b)),w_{1}(s(b),s(b),c))\\
 & =w_{1}(\varphi(a, s(b)),\varphi(r(b),s(b)),\varphi(r(b),c))\\
 & =w_{1}(\alpha(a),\beta(b),\gamma(c)),
\end{align*}
which proves that
\[
w_{1}(\varphi(a,c),\alpha(a),\beta(b))=w_{2}(\alpha(a),\beta(b),\gamma(c))=w_{1}(\varphi'(a,c),\alpha(a),\beta(b))
\]
and
\[
w_{2}(\beta(b),\gamma(c),\varphi(a,c))=w_{1}(\alpha(a),\beta(b),\gamma(c))=w_{2}(\beta(b),\gamma(c),\varphi'(a,c)),
\]
since both expressions only depend on $\alpha(a)$, $\beta(b)$ and $\gamma(c)$. Hence by definition of~$\W$ we have that $\varphi(a,c)=\varphi'(a,c)$ for all $(a,c)\in A \times_{B} C$.
\end{proof}

We could actually leave out the middle equalities (the ones not involving $x$ and~$x'$) in the quasi-identity and still obtain a weakly Mal'tsev quasivariety, but the result of this procedure would be to small to include the following example, so we are not sure that it wouldn't force the quasivariety to become Mal'tsev.

\begin{example}\label{All-but-Maltsev}
The example due to Mitschke~\cite{Mitschke} of a category which is Goursat but not Mal'tsev may be modified using Proposition~\ref{Proposition-Subvariety} to yield an example of a category which is Goursat and weakly Mal'tsev but not Mal'tsev. In fact, Proposition~\ref{Proposition-Subvariety} makes it possible to construct such examples ad libitum.

Let the variety $\V$ consist of \defn{implication algebras}, i.e., $(I,\cdot)$ which satisfy
\[
\left\{\begin{aligned}
&(x y) x=x\\
&(x y) y=(y x) x \\
&x (y z)=y (x z)
\end{aligned}\right.
\]
where we write $x\cdot y=xy$. It is proved in~\cite{Hagemann-Mitschke, Mitschke} that $\V$ is Goursat, and this is easily checked using Proposition~\ref{Proposition-n-Permutable} as witnessed by the terms $w_{1}(x,y,z)=(zy)x$ and $w_{2}(x,y,z)=(xy)z$. The further quasi-identity
\[\label{quasiid}
\left.\begin{aligned}
&(ba)x=(ab)c=(ba)x'\\
&(bc)x=(cb)a=(bc)x'\end{aligned} \right\}
\quad\Rightarrow\quad x=x'
\]
determines a weakly Mal'tsev sub-quasivariety $\W$ of $\V$ by Proposition~\ref{Proposition-Subvariety}. This quasivariety certainly stays Goursat, and the counterexample given in the paper~\cite{Mitschke} still works to prove that $\W$ is not Mal'tsev.

\begin{table}
\begin{center}
\begin{tabular}{c|cccccccc}
$a$ & 1&2&1&2&1&2&1&2\\
$b$ & 1&1&2&2&1&1&2&2\\
$c$ & 1&1&1&1&2&2&2&2\\
\hline
$x$ & 1&2&2&-&1&1&1&2
\end{tabular}
\end{center}
\caption{$x$ is uniquely determined by $a$, $b$ and $c$ in $A$}\label{Table-A}
\end{table}
\begin{table}
\begin{center}
\resizebox{\textwidth}{!}
{\begin{tabular}{c|ccccccccccccccccccccccccccc}
$a$ & 1&2&3&1&2&3&1&2&3&1&2&3&1&2&3&1&2&3&1&2&3&1&2&3&1&2&3\\
$b$ & 1&1&1&2&2&2&3&3&3&1&1&1&2&2&2&3&3&3&1&1&1&2&2&2&3&3&3\\
$c$ & 1&1&1&1&1&1&1&1&1&2&2&2&2&2&2&2&2&2&3&3&3&3&3&3&3&3&3\\
\hline
$x$ & 1&2&3&2&-&-&3&-&-&1&1&3&1&2&3&3&3&-&1&2&1&2&-&2&1&2&3
\end{tabular}}
\end{center}
\caption{$x$ is uniquely determined by $a$, $b$ and $c$ in $B$}\label{Table-B}
\end{table}
Indeed, the implication algebras $A=\{1,2\}$ and $B=\{1,2,3\}$ with respective multiplication tables
\[
\left(\begin{matrix} 1&2\\ 1&1 \end{matrix}\right)
\qquad\text{and}\qquad
\left(\begin{matrix} 1&2&3\\ 1&1&3\\ 1&2&1 \end{matrix}\right)
\]
also belong to the quasivariety $\W$: given any choice of $a$, $b$ and $c$, the system of equations
\[
\left\{\begin{aligned}
&(ba)x=(ab)c\\
&(bc)x=(cb)a\end{aligned} \right.
\]
either has no solution or just one, as pictured in Table~\ref{Table-A} for the algebra~$A$ and in Table~\ref{Table-B} for $B$.

To see that the quasivariety $\W$ is not Mal'tsev, it now suffices to consider the homomorphisms $f$, $g\colon{B \to A}$ defined respectively by
\[
f(1)=f(2)=1,\quad f(3)=2
\]
and
\[
g(1)=g(3)=1,\quad g(2)=2.
\]
It is easy to check that the respective kernel relations $R$ and $S$ of $f$ and $g$ do not commute: $RS$ contains the element $(3,2)$, but not $(2,3)$, which is in $SR$.
\end{example}

\section*{Acknowledgements}
We wish to thank the referee, Julia Goedecke, Zurab Janelidze and Diana Rodelo.


\providecommand{\noopsort}[1]{}
\providecommand{\bysame}{\leavevmode\hbox to3em{\hrulefill}\thinspace}
\providecommand{\MR}{\relax\ifhmode\unskip\space\fi MR }
\providecommand{\MRhref}[2]{%
  \href{http://www.ams.org/mathscinet-getitem?mr=#1}{#2}
}
\providecommand{\href}[2]{#2}

\small\noindent Nelson Martins Ferreira\\
Departamento de Matem\'atica, Escola Superior de Tecnologia e Gest\~ao\\
Centro para o Desenvolvimento R\'apido e Sustentado do Produto\\
Instituto Polit\'ecnico de Leiria, 2411--901 Leiria, Portugal\\
\emph{martins.ferreira@ipleiria.pt}

\vspace{2.5mm}

\noindent Tim Van~der Linden\\
CMUC, Universidade de Coimbra, 3001--454 Coimbra, Portugal\\
and\\
Institut de Recherche en Math\'ematique et Physique\\
Universit\'e catholique de Louvain\\
chemin du cyclotron~2 bte L7.01.02, 1348 Louvain-la-Neuve, Belgium\\
\textit{tim.vanderlinden@uclouvain.be}

\end{document}